\documentclass{amsart}  
\usepackage{latexsym,amssymb,palatino,eulervm,color,graphicx,amsmath,amsthm}
\usepackage{amsfonts}
\usepackage{amsthm}
\usepackage{amsmath}
\usepackage{amsfonts}
\usepackage{latexsym}
\usepackage{amssymb}
\usepackage{amscd}
\usepackage[latin1]{inputenc}
\usepackage{verbatim}

\newtheorem{definition}{Definition}[section]
\newtheorem{theorem}[definition]{Theorem}
\newtheorem{lemma}[definition]{Lemma}
\newtheorem{corollary}[definition]{Corollary}
\newtheorem{proposition}[definition]{Proposition}
\theoremstyle{definition}

\newtheorem{example}[definition]{Example}

\newcommand\style{\mathcal }          

\newcommand{\B}{\style{B}}
\newcommand{\M}{\style{M}}

\newcommand\A{{\style A}}
\renewcommand{\H}{\style{H}}
\newcommand{\K}{\style K}

\newcommand\osr{{\style R}}
\newcommand\oss{{\style S}}
\newcommand\ost{{\style T}}

\newcommand\cstar{{\rm C}^*}                              
\newcommand\cstare{{\rm C}_{\rm e}^*}              
 
\newcommand\tr{ \operatorname{Tr} }

\def\mytimeA#1
 {%
   {%
      \count20=#1 %
      \count22=\count20
      \divide \count20 by 60
      \count21=\count20
      \multiply\count21 by -60
      \advance\count22 by \count21
      \ifnum\count20<12
           \def\tapm{ a.m.}%
      \else
           \def\tapm{ p.m.}%
      \fi
      \ifnum\count20=0
           \count20=12
      \fi
      \the\count20:%
      \ifnum\count22<10 0\fi
      \the\count22
   }%
 }

\usepackage{epsfig}
\usepackage{amsbsy,latexsym}
\usepackage{amsmath}
\usepackage{amssymb, mathrsfs}
\usepackage[mathscr]{eucal}
\usepackage{hyperref}
\usepackage{amsfonts}
\usepackage{amsthm}
\usepackage{amsmath}
\usepackage{amsfonts}
\usepackage{latexsym}
\usepackage{amssymb}
\usepackage{amscd}
\usepackage[latin1]{inputenc}
\usepackage{verbatim}
\usepackage{float}
\usepackage{xcolor}
\usepackage{ragged2e}
\usepackage[english]{babel}
\usepackage{enumitem}
\usepackage{multirow}
\usepackage{makecell} 
\usepackage{array}
\newcolumntype{?}{!{\vrule width 1pt}}

\newcommand{\im}{\rm{i}}

\everymath{\displaystyle}

\begin{document}
\title{Complete Order Equivalence of Spin Unitaries}
\author[D.~Farenick]{Douglas Farenick \textsuperscript{1}}
\author[F.~Huntinghawk]{Farrah Huntinghawk \textsuperscript{2}}
\author[A.~Masanika]{Adili Masanika\textsuperscript{1}}
\author[S.~Plosker]{Sarah Plosker \textsuperscript{2,1}}

\thanks{\textsuperscript{1}Department of Mathematics and Statistics, University of Regina, Regina SK S4S 0A2, Canada}
\thanks{\textsuperscript{2}Department of Mathematics and Computer Science, Brandon University,
Brandon, MB R7A 6A9, Canada}

\date{ \today}

\begin{abstract} This paper is a study of linear spaces of matrices and linear maps on matrix algebras that arise from 
\emph{spin systems}, or \emph{spin unitaries}, which are finite sets $\mathcal S$ of selfadjoint unitary matrices such that any two unitaries in $\mathcal S$
anticommute. We are especially interested in linear isomorphisms between these linear spaces of
matrices such that the matricial order within these spaces is preserved; such isomorphisms are called 
complete order isomorphisms, which might be viewed as weaker notion of unitary similarity. The main result of this paper shows that 
all $m$-tuples of anticommuting selfadjoint unitary matrices are equivalent in this sense, meaning that there exists a unital complete order
isomorphism between the unital linear subspaces that these tuples generate. We also show that the C$^*$-envelope of any operator system
generated by a spin system of cardinality $2k$ or $2k+1$ is the simple matrix algebra $\M_{2^k}(\mathbb C)$.
As an application of the main result, we show that the free spectrahedra determined by spin unitaries depend only upon the 
number of the unitaries, not upon the particular
choice of unitaries, and we give a new, direct proof of the fact \cite{helton--klep--mcculllough--schweighofer2019}
that the spin ball $\mathfrak B_m^{\rm spin}$ and max ball $\mathfrak B_m^{\rm max}$ coincide as matrix convex sets in the cases $m=1,2$.
We also derive analogous results for countable spin systems and their C$^*$-envelopes.
\end{abstract}


\maketitle
\section{Introduction}

This paper is a study of linear spaces of matrices and linear maps on matrix algebras that arise from 
\emph{spin systems}, or \emph{spin unitaries}, which are finite sets $\mathcal S$ of selfadjoint unitary matrices such that any two unitaries in $\mathcal S$
anticommute. In addition to their interest from the perspective of linear algebra,  
these linear spaces and linear maps are commonly studied in operator algebra theory
and in applications that include quantum information theory. We are especially interested in linear isomorphisms between these linear spaces of
matrices such that the matricial order within these spaces is preserved; such isomorphisms are called 
complete order isomorphisms.

We denote the algebra of $d\times d$ matrices over the field $\mathbb C$ of complex numbers by 
$\M_d(\mathbb C)$, 
the $\mathbb R$-vector space of selfadjoint complex $d\times d$ matrices by
$\M_d(\mathbb C)_{\rm sa}$, and the cone of positive semidefinite $d\times d$ matrices
by $\M_d(\mathbb C)_+$.
The unitary group in $\M_d(\mathbb C)$ is denoted by $\mathcal U_d$, and
$\tr$ denotes the canonical trace on $\M_d(\mathbb C)$.

\begin{definition}\label{ss defn} A subset $\mathcal S\subset\mathcal U_d$ of unitary matrices 
is a \emph{spin system of unitaries}, or simply a \emph{spin system}, if
\begin{enumerate}
\item the cardinality of the set $\mathcal S$ is at least two, 
\item $u^*=u$ for every $u\in\mathcal S$, and 
\item $uv=-vu$ for every pair of distinct elements $u,v\in\mathcal S$.
\end{enumerate}
If a spin system $\mathcal S$ consists of just two elements, $u$ and $v$, then the matrices $u$ and $v$
are called a \emph{pair of spin unitaries}.
\end{definition}

The third of the three conditions in the definition of a spin system $\mathcal S$
indicates that no element of $\mathcal S$ is the identity matrix $1$ or its negative $-1$.
Hence, each $u\in\mathcal S$ has spectrum $\{-1,1\}$ and, by the Spectral Theorem, can be expressed as a difference 
$u=p-q$, where $p,q\in\M_d(\mathbb C)$ are projections such that $pq=qp=0$ and $p+q=1$.

The most basic example of a spin system of unitaries is afforded by the \emph{Pauli matrices}:  
\begin{eqnarray*}
\sigma_X=\begin{bmatrix}0&1\\1&0\end{bmatrix},\quad 
\sigma_Y=\begin{bmatrix}0&-\im\\\im&0\end{bmatrix},\quad\mbox{and}\quad
\sigma_Z=\begin{bmatrix}1&0\\0&-1\end{bmatrix}.
\end{eqnarray*}

\begin{definition} The subset $\mathcal P=\{\sigma_X,\sigma_Y,\sigma_Z\}\subset\M_2(\mathbb C)$ is called the \emph{Pauli spin system}.
\end{definition}

The following elementary (and, surely, well-known) result establishes some basic linear-algebraic facts about spin systems.

\begin{proposition}\label{prop:elementary} If $\mathcal S\subset\mathcal U_d$ is a spin system, then 
\begin {enumerate}
\item $d$ is an even integer,
\item $\tr(u)=0$, for every $u\in\mathcal S$, 
\item $\tr(uv)=0$ for all $u,v\in\mathcal S$ with $v\not=u$, and
\item the elements of $\mathcal S$ are linearly independent.
\end{enumerate}
\end{proposition} 

\begin{proof} Select $u\in\mathcal S$. By definition, there is an element $v\in\mathcal U$ with $v\not= u$. Because $uv=-vu$, we deduce from
\[
\det(u)\det(v)=\det(uv)=\det(-1uv)=(-1)^d\det(uv)
\]
that $(-1)^d=1$, and so $d$ is even. 
Likewise,
$uv=-vu $ implies that $vuv=-u$, and so, using $\tr(ab)=\tr(ba)$ for all $a,b$, we have
\[
-\tr(u)=\tr(-u)=\tr(vuv)=\tr(uv^2)=\tr(u),
\]
which yields $\tr(u)=0$.

Likewise, if $u,v\in\mathcal S$ are distinct, then $uv=-vu$ implies that $\tr(uv)=-\tr(vu)$; hence, 
$\tr(uv)=-\tr(uv)$ and, thus, $\tr(uv)=0$. 

Lastly, in considering $\M_d(\mathbb C)$ as a Hilbert space with respect to the Hilbert-Schmidt inner product $\langle x,y\rangle=\tr(y^*x)$,
the linear independence of the elements of $\mathcal S$ follows from the fact (just proved) that any two 
spin unitaries $u,v\in\mathcal S$ are orthogonal.
\end{proof}

A key concept in this paper is that of an operator system of matrices, and its matricial cone of positive semidefinite matrices.

\begin{definition} An \emph{operator system of matrices}, or more simply an \emph{operator system}, 
is a linear subspace $\osr$ of $\M_d(\mathbb C)$ 
and a sequence of convex cones $\left( \M_n(\osr)_+\right)_{n\in\mathbb N}$ in the matrix algebras $\M_n\left(\M_d(\mathbb C)\right)$
such that 
\begin{enumerate}
\item $\osr$ contains
the identity matrix $1$ (sometimes denoted by $1_d$), 
\item $\osr$ contains the adjoint $x^*$ of each matrix $x\in\osr$, and
\item $X\in \M_n(\osr)_+$ if and only if $X$ is an $n\times n$ positive semidefinite matrix with entries from $\osr$.
\end{enumerate}
\end{definition}

We note that the third aspect of the definition above is equivalent to the assertion that 
$X\in \M_n(\osr)_+$ if and only if $X$ is an $nd\times nd$ positive semidefinite matrix with entries from $\mathbb C$.

The definition of operator system given above applies only to matrices, but 
it is generally sufficient for our purposes in this paper. A somewhat more general definition 
is that an operator system is a unital linear subspace $\osr$ of 
a unital C$^*$-algebra $\A$ such that $x^*\in\osr$ for every $x\in\osr$. 
Most general of all is the axiomatic definition of an operator system as a 
matrix-ordered  $*$-vector space 
possessing an Archimedean order unit \cite{Paulsen-book}.

\begin{definition} If $\osr$ and $\ost$ are operator systems, then a linear transformation $\phi:\osr\rightarrow\ost$ is
\emph{$n$-positive} if $\phi^{(n)}(X)\in \M_n(\ost)_+$ for every $X\in\M_n(\osr)_+$, where
$\phi^{(n)}:\M_n(\osr)\rightarrow\M_n(\ost)$ is the linear map defined by
\[
\phi^{(n)}\left(\left[r_{ij}\right]_{i,j=1}^n\right)=\left[ \phi(r_{ij})\right]_{i,j=1}^n.
\]
Further:
\begin{enumerate}
\item $\phi$ is \emph{unital}, if $\phi(1_\osr)=1_\ost$ (i.e., $\phi$ maps the identity to the identity);
\item $\phi$ is \emph{positive}, if $\phi$ is $n$-positive for $n=1$; and
\item $\phi$ is \emph{completely positive}, if $\phi$ is $n$-positive for every $n\in\mathbb N$.
\end{enumerate}
\end{definition}

We turn next to the notion of
isomorphism in the category of operator systems and unital completely positive linear maps. 

\begin{definition} If $\osr$ and $\ost$ are operator systems, then a linear transformation $\phi:\osr\rightarrow\ost$ is
\begin{enumerate}
\item a \emph{unital complete order embedding} if $\phi$ is a unital, linear, completely positive, and injective map such that, 
for all $X\in\M_n(\osr)$ and all $n\in\mathbb N$, we have $\phi^{(n)}(X)\in\M_n(\ost)_+$ only if $X\in\M_n(\osr)_+$, and
\item 
a \emph{unital complete order isomorphism} if $\phi$ is a unital, linear bijection in which both $\phi$ and $\phi^{-1}$ are completely positive.
\end{enumerate}
\end{definition}

To illustrate the notion of complete order isomorphism, suppose that $w\in\mathcal U_d$ is any unitary matrix
and define the linear map $\phi_w:\M_d(\mathbb C)\rightarrow\M_d(\mathbb C)$ by $\phi_w(x)=w^*xw$, for all $x\in\M_d(\mathbb C)$.
Thus, $\phi_w$ is a unital complete order automorphism of $\M_d(\mathbb C)$. In fact, every unital complete automorphism of 
$\M_d(\mathbb C)$ arises from a unitary $w$ in this way; however, if $\osr$ and $\ost$ are operator subsystems of $\M_d(\mathbb C)$, then
there may exist unital complete order isomorphisms of $\osr$ and $\ost$ that do not arise from a unitary similarity transformation $\phi_w$.

\begin{definition} Suppose that $\mathcal S$ is a spin system of unitaries.
\begin{enumerate}
\item The \emph{spin operator system} generated by $\mathcal S$ is the linear space $\mathcal O_{\mathcal S}\subseteq\M_d(\mathbb C)$
defined by
\[
\mathcal O_{\mathcal S}=\mbox{\rm Span}\left\{1,u\,|\,u\in\mathcal S\right\}.
\]
\item The \emph{spin operator algebra} generated by $\mathcal S$ is the C$^*$-subalgebra $\mathcal A_{\mathcal S}\subseteq\M_d(\mathbb C)$
defined by
\[
\mathcal A_{\mathcal S}=\mbox{\rm Alg}\left(\mathcal O_{\mathcal S}\right).
\]
\end{enumerate}
\end{definition}

Because the elements of
a spin system $\mathcal S$ have trace zero, 
so does every linear combination of such elements; hence, the identity matrix
$1$ is linearly independent of $\mathcal S$.
Further,
the Pauli system $\mathcal P$ determines a $4$-dimensional operator system in the
$4$-dimensional matrix algebra $\M_2(\mathbb C)$, implying that
\[
\mathcal O_{\mathcal P}=\mathcal A_{\mathcal P}=\M_2(\mathbb C).
\]

Criteria for when two matrices are unitarily equivalent were given by a classical result of Specht \cite{specht1940}, as well as by others subsequently (see
\cite{shapiro1991} for a good survey). 
For operator systems, the concept of unital complete order isomorphism is weaker than the notion of unitary equivalence, which makes the following definition
of interest.
 
\begin{definition} A $k$-tuple $x=(x_1,\dots,x_k)$ of
matrices $x_j\in\M_{d}(\mathbb C)$ is \emph{completely order equivalent} to a $k$-tuple
$y=(y_1,\dots,y_k)$ of matrices 
 $y_j\in\M_\ell(\mathbb C)$ 
if there exists a unital complete order isomorphism 
\[
\phi:\mathcal O_x\rightarrow\mathcal O_y,
\]
where
$\mathcal O_x=\mbox{\rm Span}\{1_d,x_1,x_1^*,\dots,x_k,x_k^*\}$
and $\mathcal O_y=\mbox{\rm Span}\{1_\ell, y_1,y_1^*,\dots,y_k,y_k^*\}$, and $\phi(x_j)=y_j$ for $j=1, \dots, k$.
\end{definition} 

The definition of complete order equivalence of $k$-tuples of matrices
is related to that of interpolation by completely positive maps \cite{ambrozie--gheondea2015,li--poon2011}; however, a key difference
in our interpretation is that we require the interpolating maps to be complete order isomorphisms, not just completely positive maps.

We use the notation 
\[
x\simeq_{\rm ord}y
\]
to indicate that the $k$-tuples $x=(x_1,\dots,x_k)$ and
$y=(y_1,\dots,y_k)$ are completely order equivalent, and the notation 
\[
x\simeq_{\mathcal U} y
\]
to denote that the $k$-tuples $x=(x_1,\dots,x_k)$ and
$y=(y_1,\dots,y_k)$ are completely order equivalent via a unitary similarity transformation $x_j\rightarrow w^*x_j w$ for some unitary matrix $w$.

More generally, if $\osr$ and $\ost$ are operator systems of matrices in $\M_d(\mathbb C)$ and $\M_\ell(\mathbb C)$, respectively, 
then the notation 
$\osr\simeq_{\rm ord}\ost$ denotes the existence of 
a unital complete order isomorphism $\phi:\osr\rightarrow\ost$, while the notation
$\osr\simeq_{\mathcal U}\ost$ indicates that $\osr$ and $\ost$ are unitarily equivalent
(i.e., $\ell=d$ and there exists a unital complete order isomorphism $\phi:\osr\rightarrow\ost$ of the form $\phi(x)=w^*xw$, for some unitary $w\in\mathcal U_d$).

The following example, whose details we defer to the next section of this paper, is 
a good illustration of the information captured by these notions of equivalence.

\begin{example}\label{eg} If $u\in \mathcal U_{d_1}$ and $v\in \mathcal U_{d_2}$ are selfadjoint 
unitary matrices such that neither of them is a scalar multiple of the identity, then 
\begin{enumerate}
\item $u\simeq_{\rm ord} v$, and
\item $u\simeq_{\mathcal U} v$ if and only if $d_1=d_2$.
\end{enumerate}
\end{example}

The properties of anticommuting selfadjoint unitaries are not particular to matrices; indeed, these properties may
be present in arbitrary complex associative unital algebras with a positive involution $*$. 
With this understanding in mind, it is useful
to consider the most abstract form of a spin system: the universal algebra
\cite[\S II.8.3]{Blackadar-book}
that is defined purely from algebraic (rather than spatial) relations.

\begin{definition} Let $\mathcal G=\{\mathfrak u_1,\dots,\mathfrak u_m\}$ be a set of symbols and let
$\Omega$ be the set of relations
\[
\Omega=\{\mathfrak u_j^*=\mathfrak u_j\mbox{ and }\mathfrak u_j^2=1, \;\forall\,j,\mbox{ and }
\mathfrak u_i\mathfrak u_j + \mathfrak u_j \mathfrak u_i = 0, \;\forall\,j\not= i\}.
\]
\begin{enumerate}
\item
The universal C$^*$-algebra $\A_{{\rm spin}(m)}$ generated by $\mathcal G$ subject to the 
relations $\Omega$ is called the \emph{universal algebra of $m$ spin unitaries}.
\item The operator subsystem $\mathcal O_{{\rm spin}(m)}$ generated by $\mathcal G$ 
is called the \emph{universal operator system of $m$ spin unitaries}.
\end{enumerate}
\end{definition}

Universality leads immediately to the following result.

\begin{theorem}\label{u spin sys} If $\mathcal S=\{u_1,\dots,u_m\}\subset \mathcal U_d$ is a spin system, then 
there exists a unital completely positive linear map $\phi:\mathcal O_{{\rm spin}(m)}\rightarrow\mathcal O_\mathcal S$ such that 
$\phi(\mathfrak u_j)=u_j$, for every $j=1,\dots,m$.
\end{theorem}

\begin{proof}
As a universal C$^*$-algebra, the algebra $\A_{{\rm spin}(m)}$ has the property that, whenever 
$\mathcal S=\{u_1,\dots,u_m\}\subset \mathcal U_d$ are spin unitaries, there exists a C$^*$-algebra homomorphism
$\pi:\A_{{\rm spin}(m)}\rightarrow\M_d(\mathbb C)$ such that $\pi(\mathfrak u_j)=u_j$, for $j=1,\dots,m$.
Consequently, the linear map $\phi=\pi_{\vert\mathcal O_{{\rm spin}(m)}}$ is a unital completely positive linear map
of $\mathcal O_{{\rm spin}(m)}$ onto $\mathcal O_\mathcal S$ such that $\phi(\mathfrak u_j)=u_j$, for every $j=1,\dots,m$.
\end{proof}

Universal C$^*$-algebras are normally large objects; however, $\A_{{\rm spin}(m)}$ is a finite-dimensional
C$^*$-algebra. Indeed, $\A_{{\rm spin}(m)}$ is spanned by the identity $1$ and all products (of which there are finitely many) of the form
$\mathfrak u_{j_1}\mathfrak u_{j_2}\cdots\mathfrak u_{j_\ell}$, where $\ell\leq m$ and $j_1<j_2<\cdots<j_\ell$ \cite[Chapter 3]{Pisier-book}.

Our main results of the paper are the following theorem and its corollaries.

\begin{theorem}\label{main result}
If $\mathcal S\subset \mathcal U_d$ is a spin system of cardinality $m$, then 
\[
\mathcal O_\mathcal S \simeq_{\rm ord} \mathcal O_{{\rm spin}(m)}.
\]
More precisely, the unital completely positive linear map $\phi:\mathcal O_{{\rm spin}(m)}\rightarrow\mathcal O_\mathcal S$
in Proposition \ref{u spin sys} is a unital complete order isomorphism.
\end{theorem}

\begin{corollary}\label{main result cor1} If $u=(u_1,\dots,u_m)$ and $v=(v_1,\dots,v_m)$ are $m$-tuples of spin unitary matrices $u_j\in\mathcal U_{d_1}$, 
$v_k\in\mathcal U_{d_2}$, then
$u\simeq_{\rm ord} v$. 
\end{corollary}

We say that a tuple $x=(x_1,\dots,x_m)$ of $d\times d$ matrices is \emph{irreducible} if the only $d\times d$
matrices that commute with each $x_j$ are scalar multiples of the identity matrix.

\begin{corollary}\label{main result cor3} If $\M_d(\mathbb C)$ contains an irreducible $m$-tuple
$u=(u_1,\dots,u_m)$ of spin unitaries, then every $m$-tuple $v=(v_1,\dots,v_m)$ of $d\times d$ spin unitaries is also irreducible and
$u\simeq_{\mathcal U}v$.
\end{corollary}

\begin{corollary}\label{main result cor2} If $u,v,w\in\mathcal U_d$ are anticommuting selfadjoint unitary matrices, then there exists a $k\in\mathbb N$
and subspace $\mathcal L\subseteq\mathbb C^2\otimes\mathbb C^k$ such that
$u$, $v$, and $w$ are compressions to $\mathcal L$ of, respectively, $\sigma_X\otimes 1_k$, $\sigma_Y\otimes 1_k$, and $\sigma_Z\otimes 1_k$. 
\end{corollary}

If one has an operator system of matrices $\osr\subseteq\M_d(\mathbb C)$ 
and a unital complete order embedding $\phi:\osr\rightarrow\A$ into some unital C$^*$-algebra $\A$, then
the C$^*$-subalgebra $\cstar\left(\phi(\osr)\right)$ of $\A$ generated by $\phi(\osr)$ need not be 
isomorphic to the C$^*$-subalgebra $\cstar(\osr)$ of $\M_d(\mathbb C)$ generated by $\osr$. Because, in the category of operator
systems, we do not distinguish between $\osr$ and any unital complete order isomorphic copy of $\osr$, a single
operator system $\osr$ can, in principle, generate many non-isomorphic C$^*$-algebras. However, there
always exists ``smallest" such algebra, which is known as the C$^*$-envelope
of $\osr$.

\begin{theorem}[Hamana]\label{hamana}{\rm (\cite{hamana1979b,Paulsen-book})}
If $\osr$ is an operator system, then there exists a unital C$^*$-algebra $\mathcal A_{\rm e}$ and unital complete order embedding
$\iota_{\rm e}:\osr\rightarrow\A_{\rm e}$ such that $\iota_{\rm e}(\osr)$ generates $\A_{\rm e}$ and such that, if $\phi:\osr\rightarrow\A$ is any
unital complete order embedding of $\osr$ into a unital C$^*$-algebra $\A$, then there is a surjective C$^*$-algebra homomorphism
$\pi:\A\rightarrow\A_{\rm e}$ such that $\iota_{\rm e}=\pi\circ\phi$.
\end{theorem}

The algebra $\A_e$ in Hamana's Theorem is unique up to isomorphism; thus, we denote $\A_{\rm e}$ by $\cstare(\osr)$ and
say that $\cstare(\osr)$ is the \emph{C$^*$-envelope} of $\osr$.

Using our main result, Theorem \ref{main result}, we shall also prove the following theorem.

\begin{theorem}\label{cstare} For every $k\in\mathbb N$, 
$\cstare(\mathcal O_{{\rm spin}(2k)}) \cong \cstare(\mathcal O_{{\rm spin}(2k+1)}) \cong \M_{2^k}(\mathbb C)$.
\end{theorem}

Lastly, results such as Theorem \ref{main result} and Corollary \ref{main result cor1} touch upon free convexity theory
\cite{kriel2019}. We defer the pertinent definitions and discussion until later, and simply
mention here that in the present paper 
we show that the free spectrahedra determined by spin unitaries depend only upon the number of the unitaries, not upon the particular
choice unitaries, and
we give a new, direct proof of the following result.

\begin{theorem}{\rm (\cite[Corollary 14.15]{helton--klep--mcculllough--schweighofer2019})}
The free spectrahedron $\mathfrak B_2^{\rm spin}$ and the max ball $\mathfrak B_2^{\rm max}$ coincide.
\end{theorem}

The definition of spin system in the literature (for example, \cite{arveson--price2003,biane1997}) does not usually require 
the unitaries to act upon finite-dimensional Hilbert spaces. Without finite dimensionality as a consideration, we may
allow our spin systems $\mathcal S$ to have infinitely many elements. We shall consider this situation, in a brief concluding section, 
for the case where $\mathcal S$ is countably infinite.

\section{Spin Pairs}

It is instructive to begin with smallest of all spin systems: those that consist of just two elements.

\begin{definition} A \emph{spin pair} is a spin system $\oss\subset\mathcal U_d$ of cardinality $2$.
\end{definition}

Thus, any two anticommuting selfadjoint unitary matrices form a spin pair.

\begin{proposition}\label{pair} If $u\in\mathcal U_d$ is a selfadjoint unitary matrix of trace zero, then 
a selfadjoint unitary $v\in\mathcal U_d$ anticommutes with
$u$ if and only if there exist $y\in\mathcal U_d$ and $w\in \mathcal U_n$, where $n={{d}/{2}}$, such that
\[
yuy^*=\left[ \begin{array}{cc} 1_n& 0_n \\ 0_n & -1_n \end{array}\right] 
\mbox{ and }
yvy^*=\left[ \begin{array}{cc} 0_n& w \\ w^* & 0_n \end{array}\right].
\]
\end{proposition}

\begin{proof} By the Spectral Theorem, $u$ is unitarily equivalent to
a diagonal matrix in which the first $n$ entries of this diagonal matrix
are $1$ and the remaining $n$ entries are $-1$. Thus,
$yuy^*=\left[ \begin{array}{cc} 1_n& 0_n \\ 0_n & -1_n \end{array}\right] $, for some unitary matrix $y$.

Let $\tilde u=yuy^*$ and $\tilde v= yvy^*$; thus, $\tilde u$ and $\tilde v$ form a spin pair. Expressing $\tilde v$ as a $2\times 2$ matrix of $n\times n$ matrices, 
$\tilde v$ has the form $\tilde v = \left[ \begin{array}{cc} z_{11}& w \\ w^* & z_{22} \end{array}\right]$ for some matrices $z_{11},z_{22},w\in\M_n(\mathbb C)$ such that $z_{11}$ and $z_{22}$ are
selfadjoint. At the matrix level, the anticommutation relation for $\tilde u$ and $\tilde v$ leads to 
\[
\left[ \begin{array}{cc} z_{11}& w \\ -w^* & -z_{22} \end{array}\right]=\left[ \begin{array}{cc} -z_{11}& w \\ -w^* & z_{22} \end{array}\right],
\]
and so $z_{11}=z_{22}=0$. Under this condition on the diagonal blocks of $\tilde v$, we obtain $\tilde v^2=1_d$ if and only if $ww^*=w^*w=1_n$.

Therefore, $uv=-vu$ if and only if $u$ and $v$ have, via some unitary $y\in\mathcal U_d$, the $2\times 2$ block-matrix structure indicated in the statement of the proposition, for some 
unitary $w\in\mathcal U_n$.
\end{proof}

\begin{corollary} If $u\in\mathcal U_d$ is a selfadjoint unitary matrix of trace zero, then $d$ 
is an even integer and the set of all $v\in\mathcal U_d$ that anticommute with $u$
forms a path-connected topological space homeomorphic to the unitary group $\mathcal U_{\frac{d}{2}}$.
\end{corollary} 

\begin{proof} By Proposition \ref{pair}, $d$ is an even integer and there is a unitary $y\in\mathcal U_d$ 
for which $yuy^*=\left[ \begin{array}{cc} 1_n& 0_n \\ 0_n & -1_n \end{array}\right] $.
The function 
\[
F:\mathcal U_{\frac{d}{2}}\rightarrow\{v\in\mathcal U_d\,|\,uv=-vu\}
\]
defined by $F(w)= y^* \left[ \begin{array}{cc} 0_n& w \\ w^* & 0_n \end{array}\right]y$ is a homeomorphism. Because the unitary group 
$\mathcal U_{\frac{d}{2}}$ is path connected, so is the set
$\{v\in\mathcal U_d\,|\,uv=-vu\}$.
\end{proof}

Even though Example \ref{eg}(1) is a consequence of our main result, Theorem \ref{main result}, 
it is worthwhile to make note of the following simple and direct proof.

 \begin{proposition}\label{egger} If $u\in \mathcal U_{d_1}$ and $v\in \mathcal U_{d_2}$ are 
 selfadjoint unitary matrices such that neither of them is a scalar multiple of the identity, then 
\begin{enumerate}
\item $u\simeq_{\rm ord} v$, and
\item $u\simeq_{\mathcal U} v$ if and only if $d_1=d_2$.
\end{enumerate}
\end{proposition} 

\begin{proof} By Proposition \ref{pair} and the Spectral Theorem, there are unitaries $w_j\in\mathcal U_{d_j}$ such that 
\[
w_1^*uw_1=\left[\begin{array}{cc} 1_{n_1} & 0_{n_1} \\ 0_{n_1} & -1_{n_1} \end{array}\right] \mbox{ and } 
w_2^*vw_2=\left[\begin{array}{cc} 1_{n_2} & 0_{n_2} \\ 0_{n_2} & -1_{n_2} \end{array}\right],
\]
where each $n_j=\frac{d_j}{2}$. As conjugation by the unitaries $w_j$ preserves both unitary and complete order equivalence,
we may assume without loss of generality that both $u$ and $v$ are these indicated $2\times 2$ matrices of $n_j\times n_j$ matrices.
In this regard, it is clear that $u\simeq_{\mathcal U} v$ if and only if $d_1=d_2$.

More generally, to prove that $u\simeq_{\rm ord} v$, we must prove that the linear isomorphism 
\[
\phi:\mbox{Span}\{1_{d_1},u\}\rightarrow \mbox{Span}\{1_{d_1},v\}
\]
in which $\phi(1_{d_1})=1_{d_2}$ and $\phi(u)=v$ is a complete order isomorphism. To this end, we identify the matrix space $\M_n(\osr)$ with the
vector space tensor product $\M_n(\mathbb C)\otimes \osr$, where $\osr$ is an operator system. In particular, if
\[
\osr_1= \mbox{Span}\{1_{d_1},u\} \;\mbox{ and }\; \osr_2=\mbox{Span}\{1_{d_2},v\},
\]
then
\[
\M_n(\osr_j)=\left\{ a \otimes 1_{d_j} + b \otimes u_j  \,|\, a,b\in\M_{n}(\mathbb C)\right\} ,
\]
where $u_1=u$ and $u_2=v$ above. That is,
\[
\M_n(\osr_j)=\left\{  \left[\begin{array}{cc} (a+b)\otimes 1_{n_j} & 0 \\ 0 & (a-b)\otimes 1_{n_j} \end{array}\right] \,|\, a,b\in\M_{n}(\mathbb C)\right\} .
\]
Therefore, to show that $\phi$ is a unital complete order isomorphism, we much show, for all selfadjoint $a,b\in\M_n(\mathbb C)$ and all $n\in\mathbb N$, that
\[
\left[\begin{array}{cc} (a+b)\otimes 1_{n_1} & 0 \\ 0 & (a-b)\otimes 1_{n_1} \end{array}\right] 
\] 
is positive semidefinite if and only if 
\[
\left[\begin{array}{cc} (a+b)\otimes 1_{n_2} & 0 \\ 0 & (a-b)\otimes 1_{n_2} \end{array}\right] 
\]
is positive semidefinite.
This bi-implication above, 
which clearly depends only on $a$ and $b$ but not upon $u$ and $v$, shows that $\phi$ is a unital complete order isomorphism.
\end{proof}

The next theorem is the main result of this section. In preparation for its statement, 
we recall the definition of the numerical range, or field of values, of a matrix. For reasons that will be apparent in our
discussion of matrix convexity, our notation and terminology for the
numerical range (below) is slightly different from the traditional notation and terminology.

\begin{definition}
The \emph{spatial numerical range} of a matrix $x\in\M_d(\mathbb C)$ is the set
\[
W_{\rm s}^1(x)=\{\langle x\xi,\xi\rangle\,|\,\xi\in\mathbb C^d,\,\|\xi\|=1\}.
\]
\end{definition}

By a simple direct computation, the spatial
numerical range of the $2\times 2$ matrix $\left[\begin{array}{cc}0&2\\0&0\end{array}\right]$ is the closed disc of radius $1$,
centered at $0\in\mathbb C$.

\begin{theorem}\label{spin pairs thm} Let $g=\left[\begin{array}{cc}0&2\\0&0\end{array}\right]\in\M_2(\mathbb C)$.
If $u,v\in\M_d(\mathbb C)$ are anticommuting selfadjoint unitary
matrices and if $x=u+iv$, then:
\begin{enumerate} 
\item $x^2=0$ and $\|x\|=2$;
\item the spatial numerical range of $x$ is the closed unit disc;
\item $x$ is completely order equivalent to $g$, and
\item $x$ is unitarily equivalent to $\displaystyle\bigoplus_1^{n}g$, where $n=d/2$.
\end{enumerate}
\end{theorem} 

\begin{proof} 
The computation 
\[
x^2=u^2+iuv+ivu+i^2v^2=1+i(uv-uv)-1=0
\]
shows that $x$ is a nilpotent of order $2$,
while the equations
\[
x^*x=2(1+iuv) \mbox{ and } (uv)^2=-1
\]
show that the eigenvalues of $uv$ are $\pm i$ and thus the eigenvalues of $x^*x$ are $0$ and $4$, making the norm of $x$
(i.e., the square root of the spectral radius of $x^*x$) equal to $2$. 
 
By Proposition \ref{pair}, there exist 
$y\in\mathcal U_d$ and $w\in \mathcal U_n$, where $n={{d}/{2}}$, such that
\[
yuy^*=\left[ \begin{array}{cc} 1_n& 0_n \\ 0_n & -1_n \end{array}\right] 
\mbox{ and }
v=y^* \left[ \begin{array}{cc} 0_n& w \\ w^* & 0_n \end{array}\right]y .
\]
Hence, $x=u+iv$ is unitarily equivalent to $\tilde x=\left[ \begin{array}{cc} 1 & iw \\ iw^* & -1 \end{array}\right] $.
Because the norm, numerical range, and complete order equivalence are invariant under unitary equivalence,
we may assume without loss of generality that $x=\tilde x$.

Consider the unitary matrix $h=\sqrt{\frac{1}{2}}\left[ \begin{array}{cc} 1 & -1 \\ iw^* & iw^* \end{array}\right]\in\mathcal U_d$, and observe that
\[
\begin{array}{rcl}
h^*xh &=& 
\frac{1}{2}\left[ \begin{array}{cc} 1 & -iw \\ -1 &-iw \end{array}\right]\,
\left[ \begin{array}{cc} 1 &iw \\ iw^* &-1 \end{array}\right]\,
\left[ \begin{array}{cc} 1 & -1 \\ iw^* & iw^* \end{array}\right] \\ && \\
&=& -2 \left[ \begin{array}{cc} 0_n & 1_n \\0_n & 0_n\end{array}\right] \\ && \\
&=& 1_n\otimes \left[ \begin{array}{cc} 0 & -2 \\0 & 0\end{array}\right] \\ && \\
&\simeq& \displaystyle\bigoplus_1^n \left[ \begin{array}{cc} 0 & 2 \\0 & 0\end{array}\right].
\end{array}
\]
Hence, $x$ is unitarily equivalent to 
a direct sum of $n$ copies of the $2\times 2$ complex matrix $g=\left[\begin{array}{cc}0&2\\0&0\end{array}\right]$.
This implies that $g$ and $x$ have the same numerical range--namely, the closed unit disc in the complex plane. 

It remains to show that $x$ and $g$ are completely order equivalent. This follows immediately from the observation that $g$
and $\displaystyle\bigoplus_1^n g$ are completely order equivalent.
\end{proof}

As a consequence of Theorem \ref{spin pairs thm}, we recover
a fact observed in \cite{arveson--price2003}:

\begin{corollary}\label{sp pr ca} The C$^*$-algebra generated by any spin pair is $\M_2(\mathbb C)$.
\end{corollary}

\begin{proof} If $u,v\in\mathcal U_d$ form a spin pair, then
Theorem \ref{spin pairs thm} shows that $x\simeq_{\mathcal U}\displaystyle\bigoplus_1^n g$, where $n=d/2$, $x=u+iv$,
and $g=\left[\begin{array}{cc}0&2\\0&0\end{array}\right]\in\M_2(\mathbb C)$. Thus, the C$^*$-algebra generated by $x$
is isomorphic to the C$^*$-algebra generated by $g$, namely $\M_2(\mathbb C)$.

On the one hand, because $x=u+iv\in\cstar(u,v)$, we deduce that $\cstar(x)\subseteq\cstar(u,v)$. 
On the other hand, $u=\frac{1}{2}(x+x^*)$ and $v=\frac{1}{2i}(x-x^*)$
imply that $u,v\in\cstar(x)$, whence $\cstar(u,v)\subseteq\cstar(x)$.
\end{proof} 

\section{Complete Order Equivalence of Spin Unitaries}

The main result, Theorem \ref{main result}, is restated and proved below.

\begin{theorem}\label{main result proof}
If $\mathcal S\subset \mathcal U_d$ is a spin system of cardinality $m$, then 
the unital completely positive linear map $\phi:\mathcal O_{{\rm spin}(m)}\rightarrow\mathcal O_\mathcal S$
in Theorem \ref{u spin sys} is a unital complete order isomorphism and, hence,
\[
\mathcal O_\mathcal S \simeq_{\rm ord} \mathcal O_{{\rm spin}(m)}.
\]
\end{theorem} 

\begin{proof} Let $\mathcal S=\{u_1,\dots,u_m\}$. Theorem \ref{u spin sys} asserts that there exists a 
unital completely positive linear map $\phi:\mathcal O_{{\rm spin}(m)}\rightarrow\mathcal O_\mathcal S$ such that 
$\phi(\mathfrak u_j)=u_j$, for every $j=1,\dots,m$.
Because $\phi$ is a surjective linear map of vector spaces of equal finite dimension, it is an invertible linear transformation. Thus,
we need only show that its linear inverse, $\phi^{-1}$, is completely positive. 

Because the universal C$^*$-algebra $\A_{{\rm spin}(m)}$ is finite-dimensional \cite[Chapter 3]{Pisier-book}, there exist
$n\in\mathbb N$ and a unital C$^*$-algebra $\A\subseteq \M_n(\mathbb C)$ such that $\A_{{\rm spin}(m)}$ and $\A$
are isomorphic C$^*$-algebras \cite[\S 5.4]{Farenick-book1}. 
Thus, without loss of generality, we may assume that $\A_{{\rm spin}(m)}$ is a unital C$^*$-subalgebra of 
$\M_n(\mathbb C)$. Hence, the unital completely positive linear map $\phi$, when considered as a map into $\M_d(\mathbb C)$, has an
extension to a completely positive linear map $\Phi:\M_n(\mathbb C)\rightarrow\M_d(\mathbb C)$, by the Arveson Extension Theorem \cite[Theorem 7.5]{Paulsen-book}.
Therefore, by the Stinespring-Kraus-Choi Theorem \cite[Proposition 4.7]{Paulsen-book}, 
there are $\ell$ linear transformations $a_k: \mathbb C^d\rightarrow \mathbb C^n$ such that
$\Phi(z)=\displaystyle\sum_{k=1}^\ell a_k^*za_k$,
for every $z\in\M_n(\mathbb C)$. In particular, 
\begin{equation}\label{skc}
\phi(x)=\displaystyle\sum_{k=1}^\ell a_k^*xa_k,
\end{equation}
for every 
$x\in\mathcal O_{{\rm spin}(m)}$.

Because the canonical trace functional $\tr$ on the matrix algebra $\M_n(\mathbb C)$ 
induces an inner product on $\M_n(\mathbb C)$, the operator system $\mathcal O_{{\rm spin}(m)}$
is a Hilbert subspace of $\M_n(\mathbb C)$. Therefore, via the trace as an inner product, 
two matrices $y_1,y_2\in\mathcal O_{{\rm spin}(m)}$
are equal (i.e., $y_1=y_2$) if and only if $\tr(xy_1)=\tr(xy_2)$ for every matrix $x\in\mathcal O_{{\rm spin}(m)}$. We shall apply this criterion for the equality of matrices 
in what follows. To clarify notation, we shall denote the trace function on $\M_k(\mathbb C)$, for a given $k$, by $\tr_k$.

Select any $x,y\in\A_{{\rm spin}(m)}$. Thus,
\[
x=\sum_{j=1}^m \alpha_j\mathfrak u_j
\,\mbox{ and }\,
y=\sum_{j=1}^m \beta_j\mathfrak u_j
\]
for some uniquely determined scalars $\alpha_s$ and $\beta_t$. As matrices in $\M_n(\mathbb C)$, and by
using that facts that
each $\mathfrak u_j^2=1_n$  and the trace of any pair of
anticommuting matrices is $0$, we see that
\[
\tr_n(xy)=n\sum_{j=1}^m \alpha_j\beta_j.
\]
Similarly,
\begin{equation}\label{trace}
\tr_d\left(\phi(x)\phi(y)\right)=d\sum_{j=1}^m \alpha_j\beta_j  
 =\frac{d}{n}\tr_n(xy).
\end{equation}
At this point we can invoke \cite[Theorem 2.2]{nayak--sen2007} (suitably modified for operator systems)
to deduce that $\phi^{-1}$ is completely positive; however, 
owing to the complexities of the proof of that result, it
is preferable to argue directly (which we do below)
that $\phi^{-1}$ is completely positive.

In using the Stinespring-Kraus-Choi representation of $\phi$ in (\ref{skc}), the trace equation (\ref{trace})
becomes
\begin{equation}\label{3}
\begin{array}{rcl}
\tr_n(xy)&=&
\frac{n}{d}
\tr_d\left(\phi(x)\phi(y)\right)
=\frac{n}{d}  \tr_d\left( \displaystyle\sum_{i=1}^\ell\displaystyle\sum_{j=1}^\ell a_i^*xa_ia_j^*ya_j \right) \\
&& \\ 
&=&\frac{n}{d}  \tr_n\left( x\displaystyle\sum_{i=1}^\ell\displaystyle\sum_{j=1}^\ell a_ia_j^*ya_ja_i^*\right).
\end{array}
\end{equation}
Fixing $y$ and allowing $x$ to vary through all of $\mathcal O_{{\rm spin}(m)}$, equation (\ref{3}) above implies that
\[
y =\frac{n}{d}\displaystyle\sum_{i=1}^\ell\displaystyle\sum_{j=1}^\ell a_ia_j^*ya_ja_i^*.
\]
Therefore, if $\tilde\psi:\M_d(\mathbb C)\rightarrow\M_n(\mathbb C)$ is the completely positive linear map
\[
\tilde\psi(z)=\frac{n}{d}\sum_{i=1}^\ell a_iza_i^*,
\]
for $z\in\M_d(\mathbb C)$, then $\psi\circ\phi$ is the identity map on $\mathcal O_{{\rm spin}(m)}$.  
Define
\[
\psi:\mathcal O_\mathcal S\rightarrow\M_n(\mathbb C)
\] 
to be the restriction of the completely positive linear
map $\tilde\psi$ to the operator system $\mathcal O_\mathcal S$; thus, $\psi$ is a 
completely positive left inverse of $\phi$. However, because left invertible linear maps between finite-dimensional vector
spaces of equal dimension are automatically invertible, we deduce that $\psi=\phi^{-1}$, implying that $\phi^{-1}$
is completely positive.
Hence, $\mathcal O_{{\rm spin}(m)}\simeq_{\rm ord}\mathcal O_\mathcal S$.
\end{proof}

Because complete order equivalence is a transitive relation, we immediately obtain:

\begin{corollary} If $u=(u_1,\dots,u_m)$ and $v=(v_1,\dots,v_m)$ are $m$-tuples of spin unitary matrices
$u_j\in\mathcal U_{d_1}$, $v_k\in\mathcal U_{d_2}$,
then
$u\simeq_{\rm ord} v$. 
\end{corollary}

Because an operator system $\osr\subseteq\M_d(\mathbb C)$ is closed under the adjoint operation $*$, 
the von Neumann Double Commutant Theorem \cite[Theorem I.9.1.1]{Blackadar-book}
implies that $\osr{''}=\cstar(\osr)$,
where $\mathcal X'$ denotes, for a set $\mathcal X$ of matrices, the commutant of 
$\mathcal X$ (i.e., the set of all matrices that commute with every matrix in $\mathcal X$), 
and $\mathcal X''$ denotes the commutant of $\mathcal X'$. In particular, if $\osr'=\{\lambda 1_d\,|\,\lambda\in\mathbb C\}$, then 
$\cstar(\osr)=\M_d(\mathbb C)$.

\begin{definition} A spin system $\mathcal S\subset \mathcal U_d$ is \emph{irreducible} if $\mathcal S'=\{\lambda 1_d\,|\,\lambda\in\mathbb C\}$.
\end{definition}

The following result was stated as Corollary \ref{main result cor3} in the Introduction. 

\begin{proposition}\label{main result cor3 proof} If $\M_d(\mathbb C)$ contains an irreducible $m$-tuple
$u=(u_1,\dots,u_m)$ of spin unitaries, then every $m$-tuple $v=(v_1,\dots,v_m)$ of $d\times d$ spin unitaries is also irreducible and
$u\simeq_{\mathcal U}v$.
\end{proposition}

\begin{proof} By hypothesis, there exists an irreducible $m$-tuple $u=(u_1,\dots,u_m)$ of spin unitaries. Therefore, the commutant of the
operator system
$\mathcal O_u$ is $\{\lambda 1_d\,|\,\lambda\in\mathbb C\}$, implying that $\A_u=\cstar(\mathcal O_u)=\M_d(\mathbb C)$. 

Select any other $m$-tuple $v=(v_1,\dots,v_m)$ of $d\times d$ spin unitaries. By Theorem \ref{main result proof}, there is a unital complete order isomorphism
$\phi:\mathcal O_u\rightarrow\mathcal O_v$ in wihch $\phi(u_j)=v_j$, for every $j$. Let $\psi=\phi^{-1}$, as a ucp map $\mathcal O_v\rightarrow\mathcal O_u$.,
and let $\Phi,\Psi:\M_d(\mathbb C)\rightarrow\M_d(\mathbb C)$ be ucp extensions of $\phi$ and $\psi$, respectively. Therefore, $\Psi\circ\Phi$ is a ucp
extension of $\psi\circ\phi=\mbox{id}_{\mathcal O_u}$. By Arveson's Boundary Theorem \cite[Theorem 2.11]{arveson1972} (see also 
\cite{farenick2011b}, \cite[Lemma 5.11]{kriel2019}), the irreducibility of $\mathcal O_u$ implies that $\mbox{id}_{\mathcal O_u}$ has a unique ucp 
extension to $\M_d(\mathbb C)$. Hence, $\Psi\circ\Phi=\mbox{\rm id}_{\M_d(\mathbb C)}$. In other words, $\Phi$ is a unital complete positive linear map
of $\M_d(\mathbb C)$ with a completely positive inverse, which implies (by many results \cite{Stormer-book}; e.g., Wigner's Theorem)
that $\Phi$--and, hence, $\phi$--is a unitary equivalence transformation $x\mapsto w^*xw$, for some 
$w\in\mathcal U_d$. Consequently, $v$ is also an irreducible $m$-tuple of spin unitaries and $u\simeq_{\mathcal U}v$.
\end{proof}

\section{The C$^*$-Envelope of a Spin System}

The following two consequences of Hamana's Theorem (Theorem \ref{hamana}) are of use to us.

\begin{proposition}\label{erg} Suppose that $\osr\subseteq\M_d(\mathbb C)$ is an operator system of matrices.
\begin{enumerate}
\item If $\ost\subseteq\M_{\tilde d}(\mathbb C)$ is an operator system of matrices such that $\ost\simeq_{\rm ord}\osr$, then 
$\cstare(\ost)=\cstare(\osr)$.
\item If $\cstar(\osr)=\M_d(\mathbb C)$, then $\cstare(\osr)=\M_d(\mathbb C)$.
\end{enumerate}
\end{proposition}

Note that, by the Double Commutant Theorem and Proposition \ref{erg}, 
if $\mathcal S$ is an irreducible spin system, then
$\cstar_e(\mathcal O_\mathcal S)=\M_d(\mathbb C)$. Thus, focusing upon irreducible spin systems is important. 

The next result is based on a well known construction, but we do not know of a specific reference with regards to the irreducibility of the construction, and
so a (straightforward) proof is given below. 

 \begin{lemma}\label{existence of irr ss} If $\mathcal S=\{u_1,\dots,u_m\}$ is an irreducible spin system of $d\times d$ unitary matrices,
then
\begin{equation}\label{classical}
\mathcal Q=\left\{u_j\otimes 1_2, u_m\otimes\sigma_X, u_m\otimes \sigma_Y, u_m\otimes \sigma_Z\,|\,j=1,\dots,m-1  \right\}.
\end{equation}
is an irreducible spin system in $\M_{d}(\mathbb C)\otimes\M_2(\mathbb C)=\M_{2d}(\mathbb C)$.
\end{lemma}
 
\begin{proof} Consider the set $\mathcal Q\subset\mathcal U_{2d}$ defined by
\[
\mathcal Q=\left\{u_j\otimes 1_2, u_m\otimes\sigma_X, u_m\otimes \sigma_Y, u_m\otimes \sigma_Z\,|\,j=1,\dots,m-1  \right\}.
\]
Each element of $\mathcal Q$ is a selfadjoint unitary and any two distinct elements anticommute. 
Hence, $\mathcal Q$ is a spin system. We now show that $\mathcal Q$ is an irreducible spin system.

Because every element of $\mathcal Q$ is selfadjoint, a matrix $z$ commutes with each element of $\mathcal Q$ if and only if $z^*$ commutes 
with each element of $\mathcal Q$. 
Therefore, the space of matrices commuting with the elements of $\mathcal Q$ is spanned by selfadjoint matrices. Suppose, therefore, that a selfadjoint matrix
$z\in \M_{2d}(\mathbb C)$ commutes with every element of $\mathcal Q$. Identifying
$z\in \M_{2d}(\mathbb C)$ with $ \M_2\left( \M_{d}(\mathbb C)\right)$, the selfadjoint matrix $z$ can written as
\[
 z=\begin{bmatrix} a_{11}&a_{12} \\ a_{12}^*&a_{22} \end{bmatrix},
\]
for some $a_{11},a_{12},a_{22}\in \M_{d}(\mathbb C)$. Likewise, $u_j\otimes 1_2$ and $u_m\otimes\sigma_Z$, for $j=1,\dots,m-1$, are given by
\[
\begin{bmatrix} u_j&0 \\ 0&u_j\end{bmatrix} \mbox{ and } \begin{bmatrix} u_m&0 \\ 0&-u_m \end{bmatrix}.
\]
The commutation relations $z(u_j\otimes 1_2)=(u_j\otimes 1_2)z$ and $z(u_m\otimes\sigma_Z)=(u_m\otimes\sigma_Z)z$ yield
\[
\begin{bmatrix} a_{11}u_j&a_{12}u_j \\ a_{12}^*u_j&a_{22}u_j \end{bmatrix}=\begin{bmatrix} u_ja_{11}&u_ja_{12} \\ u_ja_{12}^*&u_ja_{22} \end{bmatrix}
\]
for $j=1,\dots,m-1$, and 
\[
\begin{bmatrix} a_{11}u_m&-a_{12}u_m \\ a_{12}^*u_m& -a_{22}u_m \end{bmatrix} = \begin{bmatrix} u_ma_{11}&u_ma_{12} \\ -u_ma_{12}^*&-u_ma_{22} \end{bmatrix}.
\]
Therefore, $a_{12}$ commutes with $u_j$ for $j=1,\dots,m-1$. Furthermore, 
$a_{11}$ and $a_{22}$ commute with every element of $\mathcal S$, which implies that $a_{jj}=\alpha_{jj}1_{2n}$ for some $\alpha_{jj}\in\mathbb R$.
Hence, $z$ has the form
\[
z= \begin{bmatrix} \alpha_{11}1_{2n}&a_{12} \\ a_{12}^*&\alpha_{22}1_{2n} \end{bmatrix}.
\]
Using the commutation relation $z(u_m\otimes\sigma_X)=(u_m\otimes\sigma_X)z$ and the identification 
$u_m\otimes\sigma_X=\begin{bmatrix} 0&u_m \\ u_m&0 \end{bmatrix}$, we obtain
\[
\begin{bmatrix} u_ma_{12}^*&\alpha_{22}u_m \\ \alpha_{11}u_m&u_ma_{12} \end{bmatrix} 
= \begin{bmatrix} a_{12}u_m&\alpha_{11}u_m \\ \alpha_{22}u_m&a_{12}^*u_m \end{bmatrix},
\]
which yields $\alpha_{11}=\alpha_{22}$ and $(a_{12}+a_{12}^*)u_m=u_m(a_{12}+a_{12}^*)$. 
Because $a_{12}+a_{12}^*$ also commutes with each $u_j$ for $j=1,\dots,m-1$, we conclude that
$a_{12}+a_{12}^*=\lambda 1_{d}$, for some $\lambda\in\mathbb R$. In setting $\alpha=\alpha_{11}$, 
the commutation relation $z(u_m\otimes\sigma_Y)=(u_m\otimes\sigma_Y)z$
yields
\[
\begin{bmatrix} -\im a_{12}u_m&\alpha\im u_m \\ \alpha\im u_m&-\im a_{12}^*u_m \end{bmatrix}
=
\begin{bmatrix} -\im u_ma_{12}^*&-\alpha\im u_m \\ \alpha\im u_m& \im u_ma_{12} \end{bmatrix} .
\]
Thus, $a_{12}-a_{12}^*$ commutes with $u_m$ and with  each $u_j$ for $j=1,\dots,m-1$, we 
conclude that $a_{12}-a_{12}^*=\mu 1_{d}$ for some scalar $\mu\in\mathbb R$, and so
$a_{12}=(\lambda+ i \mu)1_{d}$, which is a scalar multiple of the identity matrix. 
Therefore, $a_{12}$ commutes with every matrix. However, because $a_{12}$ both
commutes and anticommutes with $u_m$, this scalar must be zero. Hence, $z$ is a 
scalar multiple of the identity matrix, which proves that $\mathcal Q$ is an irreducible spin system.
\end{proof}

\begin{theorem}\label{envelope thm} 
$\cstare(\mathcal O_{{\rm spin}(2k)})=\cstare(\mathcal O_{{\rm spin}(2k+1)})=\displaystyle\bigotimes_1^k\M_{2 }(\mathbb C)$, for every $k\in\mathbb C$.
\end{theorem}

\begin{proof}  The C$^*$-algebra generated by any spin pair is the simple algebra
$\M_2(\mathbb C)$ (Corollary \ref{sp pr ca}), while the operator system spanned by the Pauli matrices
is $\M_2(\mathbb C)$.
As any spin pair or triple is completely order equivalent to the pair $(\sigma_X,\sigma_Y)$
or the triple $(\sigma_X,\sigma_Y,\sigma_Z)$ (by Theorem \ref{main result}), we obtain (from Proposition \ref{erg}) the 
following algebra equalities:
\[
\cstare(\mathcal O_{{\rm spin}(2)})=\cstare(\mathcal O_{{\rm spin}(3)})=\M_2(\mathbb C).
\]

Using the irreducible spin system $\mathcal Q_1=\{\sigma_X,\sigma_Y,\sigma_Z\}\subset\M_2(\mathbb C)$, the 
construction of the spin system in Lemma \ref{existence of irr ss} produces the following irreducible spin system 
$\mathcal Q_2\subset\M_2(\mathbb C)\otimes\M_2(\mathbb C)$ of $5$ elements:
\[
\begin{array}{rcl}
\mathcal Q_2&=&\{ \sigma_X\otimes 1,  \sigma_Y\otimes 1, \sigma_Z\otimes \sigma_X, \sigma_Z\otimes \sigma_Y, \sigma_Z\otimes \sigma_Z\} \\
&=& \mathcal Q_{2,-} \cup \{\sigma_Z\otimes \sigma_Z\},
\end{array}
\]
where $\mathcal Q_{2,-} =\mathcal Q_2 \setminus \{\sigma_Z\otimes \sigma_Z\}$.

Another iteration of the construction in Lemma \ref{existence of irr ss} yields an irreducible spin system $\mathcal Q_3$ of $7$ elements:
\[
\mathcal Q_3= \mathcal Q_{3,-} \cup \{\sigma_Z\otimes \sigma_Z\otimes\sigma_Z\},
\]
where
\[
\mathcal Q_{3,-}=\{  
\sigma_X\otimes 1 \otimes 1 ,
 \sigma_Y\otimes 1 \otimes 1 ,
 \sigma_Z\otimes \sigma_X \otimes 1,
 \sigma_Z\otimes \sigma_Y \otimes 1, 
 \sigma_Z\otimes \sigma_Z \otimes \sigma_X ,
\sigma_Z\otimes \sigma_Z \otimes \sigma_Y\}.
\]
In general,
\[
\mathcal Q_k= \mathcal Q_{k,-} \cup \left\{\bigotimes_1^k\sigma_Z \right\}.
\]

The key point to observe is that the elements of $\mathcal Q_{k,-}$ consist of $k$ pairs such that, in the order given by the iterative construction,
the product of each pair is a product tensor in which all factors are the identity matrix and one tensor factor is $\sigma_X\sigma_Y$. More specifically,
if
\[
\mathcal Q_{k,-}=\{w_1,w_2,w_3,w_4,\dots,w_{2k-1},w_{2k}\} \subset \bigotimes_1^k \M_2(\mathbb C),
\]
then
\[
\begin{array}{rcl}
w_1w_2&=&(\sigma_X\sigma_Y)\otimes 1 \otimes 1 \cdots \otimes 1 = i\left( \sigma_Z \otimes 1 \otimes 1 \cdots \otimes 1\right) \\
w_3w_4&=&1 \otimes (\sigma_X\sigma_Y)\otimes 1\cdots\otimes 1 = i\left( 1 \otimes \sigma_Z\otimes 1\cdots \otimes 1\right) \\
\vdots &=& \vdots \\
w_{2k-1}w_{2k}&=& 1 \otimes 1 \otimes 1 \cdots \otimes (\sigma_X\sigma_Y)= i \left(1 \otimes 1 \otimes 1 \otimes 1 \cdots \otimes \sigma_Z\right). 
\end{array}
\]
Hence,
\[
\bigotimes_1^k\sigma_Z = i^{-k} \displaystyle\prod_{j=1}^k w_{2j-1}w_{2j} \in \mbox{\rm Alg}\left(\mathcal Q_{k,-} \right) ,
\]
which shows that
\[
\cstar (\mathcal O_{\mathcal Q_{k,-}}) =\cstar (\mathcal O_{\mathcal Q_k}),
\]
for every $k\in \mathbb N$. Therefore, because $\mathcal Q_k$ is an irreducible spin system, 
\[
\cstar (\mathcal O_{\mathcal Q_{k,-}}) =\cstar (\mathcal O_{\mathcal Q_k})=\bigotimes_1^k\M_2(\mathbb C).
\]
Therefore, the C$^*$-envelopes of $\mathcal O_{\mathcal Q_{k,-}}$ and $\mathcal O_{\mathcal Q_{k}}$ are also 
$\displaystyle\bigotimes_1^k\M_2(\mathbb C)$.
 
Therefore, by replacing $\mathcal O_{{\rm spin}(2k)}$ with $\mathcal O_{\mathcal Q_{k,-}}$ and
$\mathcal O_{{\rm spin}(2k+1)}$ with $\mathcal O_{\mathcal Q_k}$ (by Theorem \ref{main result}),
we obtain
\[
\cstare(\mathcal O_{{\rm spin}(2k)})=\cstare(\mathcal O_{{\rm spin}(2k+1)})=\displaystyle\bigotimes_1^k\M_{2 }(\mathbb C).
\]
\end{proof}

\section{Free Spectrahedra and Dilations}

In this partly expository section, we apply the notions developed in this paper to examine some known results on
free spectrahedra and matrix ranges, and also prove a new result regarding the dilation of spin triples. 

\begin{definition}{\rm (\cite{kriel2019})} 
Suppose that $a=(a_1,\dots,a_m)$ is an $m$-tuple of selfadjoint $d\times d$ matrices. 
\begin{enumerate} 
\item The monic polynomial $L_a(t_1,\dots,t_m)=1_d-\displaystyle\sum_{j=1}^m t_ja_j$, in variables $t_1,\dots,t_m$, evaluated at
an $m$-tuple $h=(h_1,\dots,h_m)$ of $n\times n$ selfadjoint matrices, is the selfadjoint element $L_a(h)\in\M_d(\mathbb C)\otimes\M_n(\mathbb C)$
defined by
\[
L_a(h)=1_n\otimes 1_d - \sum_{j=1}^m h_j\otimes a_j.
\]
\item The \emph{free spectrahedron} determined by $a$ is the sequence $\mathcal D_a=\left(\mathcal D_{a,n}\right)_{n\in\mathbb N}$
of subsets
\[
\mathcal D_{a,n}=\left\{h=(h_1,\dots,h_n)\,|\,\mbox{each }h_j\in\M_n(\mathbb C)_{\rm sa} \mbox{ and }  L_a(h)\mbox{ is positive semidefinite}\right\}.
\]
\end{enumerate}
\end{definition}

The first result shows that the free spectrahedra determined by spin systems depends only upon the cardinality of the spin system, not upon the choice
of spin unitaries.

\begin{proposition}\label{fs1} If $u=(u_1,\dots,u_m)$ and $v=(v_1,\dots,v_m)$ are $m$-tuples of spin unitaries with $u_j\in\mathcal U_{d_1}$ and $v_k\in\mathcal U_{d_2}$, then
$\mathcal D_u=\mathcal D_v$.
\end{proposition}

\begin{proof}  The canonical linear bases of $\mathcal O_u$ and $\mathcal O_v$ are, respectively, $\{1_{d_1},u_1,\dots,u_m\}$ and
$\{1_{d_2},v_1,\dots,v_m\}$. In particular, by identifying $\M_n(\mathcal O_u)$ with $\M_n(\mathbb C)\otimes \mathcal O_u$,
a selfadjoint matrix $y\in \M_n(\mathcal O_u)$ is expressed as
\[
y= b_0\otimes 1_{d_1} + \sum_{j=1}^m b_j\otimes u_j,
\]
for some (uniquely determined) selfadjoint matrices $b_0,b_1,\dots,b_m\in\M_n(\mathbb C)$. Likewise, the element
\[
\tilde y = b_0\otimes 1_{d_2} + \sum_{j=1}^m b_j\otimes v_j
\]
is a selfadjoint elements of $\M_n(\mathbb C)\otimes \mathcal O_v=\M_n(\mathcal O_v)$.

Corollary \ref{main result cor1} asserts that the linear map 
$\phi:\mathcal O_u\rightarrow\mathcal O_v$ in which $\phi(1_{d_1})=1_{d_2}$ and $\phi(u_j)=v_j$, for each $j$, is a unital complete order isomorphism
of $\mathcal O_u$ and $\mathcal O_v$. Thus, the equation $\tilde y= (\mbox{\rm id}_{\M_n(\mathbb C)}\otimes\phi)[y]$ shows that
\[
b_0\otimes 1_{d_1} + \sum_{j=1}^m b_j\otimes u_j \mbox{ is positive semidefinite}
\]
if and only if
\[
b_0\otimes 1_{d_2} + \sum_{j=1}^m b_j\otimes v_j
\mbox{ is positive semidefinite}.
\]
In particular, given an $m$-tuples $h$ of selfadjoint matrices $h_j\in\M_n(\mathbb C)$,
$L_u(h)$ is positive semidefinite if and only if $L_v(h)$ is positive semidefinite. Hence, $\mathcal D_{u,n}=\mathcal D_{v,n}$,
for every $n\in\mathbb N$.
\end{proof}

In \cite{helton--klep--mcculllough--schweighofer2019}, the \emph{spin ball} $\mathfrak B_m^{\rm spin}$ is defined to be
the free spectrahedron
determined by a spin system constructed iteratively from the Pauli matrices, as in the proof of Theorem \ref{envelope thm}. In light of
Proposition \ref{fs1}, the spin ball can be defined unambiguouly as follows.

\begin{definition}[Spin Ball] The \emph{spin ball} $\mathfrak B_m^{\rm spin}$ is the free spectrahedron $\mathcal D_u$ for any
$m$-tuple $u$ of spin unitaries $u_j\in\mathcal U_d$.
\end{definition}

Free spectrahedra are easily seen to be matrix convex. Before defining matrix convexity below, note that  
the Cartesian product $\displaystyle\prod_1^m\M_n(\mathbb C)$ of $m$ copies of $\M_n(\mathbb C)$ is a unital C$^*$-algebra,
which makes the consideration of completely positive linear maps
between such spaces of interest. In particular, if $\gamma:\mathbb C^n\rightarrow\mathbb C^k$ is a linear transformation, 
then we have an induced completely
positive linear map 
\[
\Gamma:\displaystyle\prod_1^m\M_n(\mathbb C)\rightarrow\displaystyle\prod_1^m\M_k(\mathbb C)
\]
defined by $\Gamma(x)=\gamma^*\cdot x\cdot \gamma$, for all $x=(x_1,\dots,x_m)\in \displaystyle\prod_1^m\M_n(\mathbb C)$,
where
\[
\gamma^*\cdot x \cdot \gamma =\left( \gamma^* x_1 \gamma, \dots, \gamma^*x_m\gamma\right).
\]

\begin{definition}{\rm (\cite{kriel2019})} For a fixed $m\in\mathbb N$, suppose that $\mathcal K=\left( \mathcal K_n\right)_{n\in\mathbb N}$ is a 
sequence of subsets $\K_n\subseteq \displaystyle\prod_1^m\M_n(\mathbb C)$. If the sequence $\mathcal K$ has the property that
\[
\sum_{\ell=1}^t \gamma_\ell^*\cdot \Lambda_\ell\cdot\gamma_\ell \in \mathcal K_n, 
\]
for all $t\in\mathbb N$, all $\Lambda_\ell\in\mathcal K_{n_\ell}$, and 
all linear transformations $\gamma_\ell:\mathbb C^n\rightarrow\mathbb C^{n_\ell}$ such that
\[
\sum_{\ell=1}^t \gamma_\ell^* \gamma_\ell = 1_n,
\]
the $\mathcal K$ is said to be \emph{matrix convex}.
\end{definition}            

In addition to free spectrahedra, matrix ranges form another class of matrix convex sets. 

\begin{definition}[Matrix Range] If $x=(x_1,\dots,x_m)$ is an $m$-tuple of matrices $x_j\in\M_d(\mathbb C)$, then the 
\emph{matrix range} of $x$ is the sequence $W(x)=\left(W^n(x)\right)_{n\in\mathbb N}$ in which each
$W^n(x)$ is defined by
\[
W^n(x)=\left\{\phi(x)\,|\,\phi:\mathcal O_x\rightarrow\M_n(\mathbb C)\mbox{ is a ucp map}\right\},
\]
where $\mathcal O_x$ is the operator subsystem of $\M_d(\mathbb C)$ generated by $x$
and where $\phi(x)$ is the $m$-tuple of elements in $\M_n$ given by
\[
\phi(x)=\left(\phi(x_1),\dots,\phi(x_m) \right).
\]
\end{definition}

The relevance of matrix ranges to complete order equivalence
and unitary equivalence originates in the work of 
Arveson \cite{arveson1972}.  

\begin{theorem}\label{mr thm}{\rm (\cite{arveson1972,davisdon--dor-on--shalit--solel2017})}
The following statements are equivalent for 
tuples $x=(x_1,\dots,x_m)$ and $y=(y_1,\dots,y_m)$ of matrices
$x_j\in\M_{d_1}(\mathbb C)$ and $y_k\in\M_{d_2}(\mathbb C)$:
\begin{enumerate}
\item $x\simeq_{\rm ord} y$;
\item $W(x)=W(y)$.
\end{enumerate}
\end{theorem}

There are many examples of matrix ranges in the literature. The following example, which is relevant to the subject of the present paper,
can be deduced as a special case of \cite[Example 1]{farenick1992}.

\begin{example} Let $u \in\mathcal U_d$ be non-scalar selfadjoint unitary matrices. Then, for every $n\in\mathbb N$,
$h\in W^n(u)$ if and only if there exist  $a,b\in\M_n(\mathbb C)_+$ such that $a+b=1_n$ and $a-b=h$.
\end{example}

There is also a spatial version of the matrix range.

\begin{definition} If $x=(x_1,\dots,x_m)$ is an $m$-tuple of matrices $x_j\in\M_d(\mathbb C)$, then the 
\emph{spatial matrix range} of $x$ is the finite sequence
$W_{\rm s}(x)=\left(W^n(x)\right)_{n=1}^d$ in which each
$W_{\rm s}^n(x)$ is defined by
\[
W_{\rm s}^n(x)=\left\{\gamma^*\cdot x \cdot \gamma\,|\,
\gamma:\mathbb C^n\rightarrow \mathbb C^d\mbox{ is a linear isometry}\right\}.
\]
\end{definition}

Note that $W_{\rm s}^n(x)\subseteq W ^n(x)$ for every $n=1,\dots,d$. However, the spatial matrix range
lacks the strong feature of matrix convexity; indeed,  $W_{\rm s}^d(x)$ is the unitary orbit of $x$, and therefore
fails to contain 
any line segments (in the classical sense) whatsoever.

In the case of $n=1$, the spatial matrix range $W_{\rm s}^1(x)$ is better known in linear algebra as 
the (joint) numerical range of $x=(x_1,\dots,x_m)$. 
The following elementary calculation is well known to many linear algebraists.

\begin{example}
If $\sigma=(\sigma_X,\sigma_Y,\sigma_Z)$, then 
\begin{enumerate}
\item the spatial numerical range $W_{\rm s}^1(\sigma)$ is the unit Euclidean sphere in $\mathbb R^3$, and
\item the numerical range $W^1(\sigma)$ is the closed unit Euclidean ball in $\mathbb R^3$
\end{enumerate}
\end{example}

\begin{proof} Every unital positive linear functional $\phi:\M_2(\mathbb C)\rightarrow \mathbb C$ is a convex combination of linear functionals
of the form $\omega_\xi(x)=\langle x\xi,\xi \rangle$, for a unit vector $\xi\in\mathbb C^2$. Thus, $W^1(\sigma)$ is the convex hull
of
\[
W_{\rm s}^1(\sigma)=\left\{ \left(\langle \sigma_X\xi,\xi\rangle, \langle \sigma_Y\xi,\xi\rangle, \langle \sigma_Z\xi,\xi\rangle\right)
 \,|\,
\xi\in\mathbb C^2,\,\langle\xi,\xi\rangle=1\right\}.
\]
Thus, statement (2) follows by showing statement (1) holds. To obtain (1), note that because the inner product on $\mathbb C^2$
is not bilinear but sesquilinear, it is enough to compute the spatial numerical range using unit vectors of the form
$\xi=(\cos\theta) e_1 + e^{i\delta}(\sin\theta) e_2$, for all $\theta,\delta\in\mathbb R$. Since
\[
\begin{array}{rcl}
\langle\sigma_X\xi,\xi\rangle&= &2\Re(e^{i\delta}\sin\theta\,\cos\theta)=\cos\delta\sin(2\theta),
\\
\langle\sigma_Y\xi,\xi\rangle&= &2\Re(ie^{-i\delta}\sin\theta\,\cos\theta)=\sin\delta\sin(2\theta),\mbox{ and }
\\
\langle\sigma_Z\xi,\xi\rangle&=& \cos^2\theta-\sin^2\theta=\cos(2\theta),
\end{array}
\]
we obtain the spherical coordinates for the unit Euclidean sphere $S^2$ in $\mathbb R^3$.
Hence, $W_{\rm s}^1(\sigma)=S^2$.
\end{proof}

For each $m\in\mathbb N$, let $\mathbb B_m$ denote the closed unit Euclidean ball of $\mathbb R^m$. 

\begin{definition}[Max Ball]{\rm (\cite[\S 14.2.2]{helton--klep--mcculllough--schweighofer2019})}
For each $m\in\mathbb N$, the \emph{max ball} $\mathfrak B_m^{\rm max}$ is the sequence
$\mathfrak B_m^{\rm max}=\left(B_{m,n} \right)_{n\in\mathbb N}$ whereby an $m$-tuple
$h=(h_1,\dots,h_n)$ of selfadjoint $n\times n$ matrices belongs to  $B_{m,n}$ if and only if,
for every $m$-tuple $a=(a_1,\dots,a_m)$ of selfadjoint
$d\times d$ matrices and for every $d$, 
the $m$-tuple $h$ is necessarily an element $\mathcal D_{a,n}$ if $\mathcal D_{a,1}$ contains the Euclidean ball 
$\mathbb B_m$.
\end{definition} 

Combining results from \cite[\S 14.2.2]{helton--klep--mcculllough--schweighofer2019}), one obtains the following
criterion for membership in the max ball.

\begin{proposition}\label{criterion} An $m$-tuple $h=(h_1,\dots,h_n)$ of selfadjoint $n\times n$ matrices belongs to 
(some element of the sequence) $\mathfrak B_m^{\rm max}$ if and only if
\[
1_n\otimes 1_d - \sum_{j=1}^m h_j\otimes a_j \mbox{ is positive semidefinite}
\]
for for every $m$-tuple $a=(a_1,\dots,a_m)$ of selfadjoint
$d\times d$ matrices, and for every $d$, in which 
\[
W_{\rm s}^1(a)\subseteq\mathbb B_m.
\]
\end{proposition}

The max ball is not, at first glance, a free spectrahedron 
because the defining conditions for membership in $\mathfrak B_m^{\rm max}$
involve, in principle, infinitely many monic polynomials $L_a(t_1,\dots,t_m)$. However, 
in low dimensions, the max ball is a free spectrahedron, as shown by the following theorem
of Helton, Klep, McCullough, and Schweighoefer \cite[Corollary 14.15]{helton--klep--mcculllough--schweighofer2019}. The authors
of \cite{helton--klep--mcculllough--schweighofer2019} give two proofs of this result using dilation. We offer 
a third alternative below.

\begin{theorem}\label{spin ball thm} $\mathfrak B_1^{\rm spin}=\mathfrak B_1^{\rm max}$ and $\mathfrak B_2^{\rm spin}=\mathfrak B_2^{\rm max}$.
\end{theorem} 

\begin{proof} 
By Proposition \ref{fs1}, the spin ball $\mathfrak B_2^{\rm spin}=\mathfrak B_2^{\rm max}$ is the free spectrahedron determined by the Pauli matrices
$(\sigma_X,\sigma_Y)$. 
Because $\sigma_X+i\sigma_Y=\left[ \begin{array}{cc} 0&2 \\ 0&0\end{array}\right]$, the spatial numerical range $W_{\rm s}^1(\sigma_X,\sigma_Y)$ is
equal to the closed Euclidean disc $\mathbb B_2$ in $\mathbb R^2$. Hence, if
a selfadjoint pair $(h_1,h_2)\in\M_n(\mathbb C)\times\M_n(\mathbb C)$
belongs to $\mathfrak B_1^{\rm max}$, then by definition
\begin{equation}\label{xy}
1_n\otimes 1_2 - \left( h_1\otimes \sigma_X + h_2\otimes \sigma_Y\right)
\end{equation}
is postive semidefinite.
Consequently, $(h_1,h_2)$
belongs to $\mathfrak B_2^{\rm spin}$, showing that 
$\mathfrak B_2^{\rm max}\subseteq\mathfrak B_2^{\rm spin}$.

Conversely, assume that
a selfadjoint pair $(h_1,h_2)\in\M_n(\mathbb C)\times\M_n(\mathbb C)$
belongs to $\mathfrak B_2^{\rm spin}$. Thus, the matrix in equation (\ref{xy}) above is positive semidefinite. 
Select any pair $(a_1,a_2)$ of $\ell\times \ell$ selfadjoint matrices in which 
$W_{\rm s}^1(a_1,a_2)\subseteq\mathbb B_2$. As the numerical radius of $y=\frac{1}{2}(a_1+ia_2)$ is at most $1$, there exists, by Ando's Theorem \cite{ando1973},
a positive semidefinite contraction $b\in\M_\ell (\mathbb C)$ such that
\[
c=\left[ \begin{array}{cc} 1 & y \\ y^* & 1_\ell-b\end{array}\right]
\]
is positive semidefinite. Let $\psi:\M_2(\mathbb C)\rightarrow\M_\ell(\mathbb C)$ be the unital linear map in which $\psi(e_{11})=b$,
$\psi(e_{12})=y$, $\psi(e_{21})=y^*$, and $\psi(e_{22})=1_\ell-b$. The matrix $c$ above is the Choi matrix for $\psi$; 
therefore, by Choi's Criterion \cite[Theorem 3.14]{Paulsen-book}, $\psi$ is completely positive.
Now since $a_1+ia_2=\psi(\sigma_X+i\sigma_Y)$, equating real 
and
imaginary parts yields $a_1=\psi(\sigma_X)$ and $a_2=\psi(\sigma_Y)$. Hence, the 
positivity of 
\[
1_n\otimes 1_2 - \left( h_1\otimes \sigma_X + h_2\otimes \sigma_Y\right)
\]
implies the positivity of
\[
(\mbox{\rm id}_{\M_n(\mathbb C)}\otimes\psi ) \left[1_n\otimes 1_2 - \left( h_1\otimes \sigma_X + h_2\otimes \sigma_Y\right)\right ]  = 1_\ell\otimes 1_2 - \left( h_1\otimes a_1 + h_2\otimes a_2\right),
\]
That is, $(h_1,h_2)$ belongs to $\mathfrak B_1^{\rm max}$, proving that 
$\mathfrak B_2^{\rm spin}\subseteq\mathfrak B_2^{\rm max}$.

The proof that $\mathfrak B_1^{\rm spin}=\mathfrak B_1^{\rm max}$ is more straightforward, and is left to the interested reader.
\end{proof} 

It would, of course, be very interesting to know whether Theorem \ref{spin ball thm} extends to 
higher dimensions. Some evidence that this might be so is
presented in \cite{passer2019}.

\begin{definition} An $m$-tuple $y=(y_1,\dots,y_m)$ of matrices $y_j\in \M_{d_2}(\mathbb C)$ 
is a \emph{dilation} of an $m$-tuple $x=(x_1,\dots,x_m)$ of matrices $x_j\subset \M_{d_1}(\mathbb C)$ if there exists a linear isometry $w:\mathbb C^{d_1}\rightarrow\mathbb C^{d_2}$
such that 
$x_j=w^*y_jw$, for every $j$.
\end{definition} 

Put differently, $y=(y_1,\dots,y_m)$ is a dilation of $x=(x_1,\dots,x_m)$ if
there is a unitary $z\in\mathcal U_{d_2}$ such that
\[
z^*y_jz=\left[ \begin{array}{cc} x_j & * \\ * & * \end{array}\right],
\]
for every $j$.

A reinterpretation of Corollary \ref{main result cor1} leads to the following dilation result
for triples of spin unitaries.  
 
\begin{proposition} For every spin triple $u,v,w\in\mathcal U_d$, 
there exists a $k\in\mathbb N$ such that 
$(\sigma_X\otimes 1_k, \sigma_Y\otimes 1_k, \sigma_Z\otimes 1_k)$
is a dilation of $(u,v,w)$, and there exists a
$\ell\in\mathbb N$ such that
$(u\otimes 1_\ell, v\otimes 1_\ell, w\otimes 1_\ell)$
is a dilation of $(\sigma_X,\sigma_Y,\sigma_Z)$.  
\end{proposition}

\begin{proof} By Corollary \ref{main result cor1}, $(u,v,w)\simeq_{\rm ord} (\sigma_X,\sigma_Y,\sigma_Z)$.
Therefore, the triples
$(u,v,w)$ and $(\sigma_X,\sigma_Y,\sigma_Z)$ have identical matrix ranges, by Theorem \ref{mr thm}. In particular,
\[
(u,v,w)\in W^d(\sigma_X,\sigma_Y,\sigma_Z).
\]
The operator system generated by the Pauli matrices is the C$^*$-algebra $\M_2(\mathbb C)$; therefore, the inclusion
above indicates that $u=\phi(\sigma_X)$, $v=\phi(\sigma_Y)$, and $w=\phi(\sigma_Z)$, for some unital completely
positive linear map $\phi:\M_2(\mathbb C)\rightarrow\M_d(\mathbb C)$.
By the Stinespring Theorem \cite{Paulsen-book}, $\phi$ has the form $\phi(y)=\gamma^*\pi(y) \gamma$, 
for some unital representation
$\pi$ of $\M_2(\mathbb C)$ on which the representing Hilbert space has finite dimension. In other words, $\pi(y)$
is a dilation of $\phi(y)$, for every $y\in\M_2(\mathbb C)$.
Because every representation
of a full matrix algebra is unitarily equivalent to a direct sum of the identify representation, we may assume that $\gamma$
and $\pi$ are so chosen so that 
\[
\pi(y)=\bigoplus_1^k y = y\otimes 1_k,
\]
thereby implying that 
$(\sigma_X\otimes 1_k, \sigma_Y\otimes 1_k, \sigma_Z\otimes 1_k)$
is a dilation of $(u,v,w)$.

The second statement is argued in the same manner by interchanging the roles of the spin triples. 
\end{proof}

\section{Countable Spin Systems}

Our main results are linear-algebraic in nature; however, they can be applied to the study of general spin systems in physics. 

In this concluding section, $\H$ shall denote an infinite-dimensional separable complex Hilbert space, and $\B(\H)$ shall denote the algebra of all
bounded linear operators on $\H$. A \emph{spin system} is a finite or countably infinite set $\mathcal S$ of 
selfadjoint unitary operators on $\H$ such that $uv=-vu$ for any pair $u,v\in\mathcal S$. We shall assume, in this section, that
the reader is familiar with C$^*$-algebra theory (e.g., \cite{Blackadar-book}) and the general theory of operator systems (e.g., \cite{choi--effros1977,Paulsen-book}). 

We continue with the notation $\A_{{\rm spin}(m)}$ for the universal C$^*$-algebra generated by $m$ spin unitaries, and we introduce the
notation $\A_{{\rm spin}(\aleph_0)}$ for the universal C$^*$-algebra generated by a countably-infinite number of spin unitaries. The existence of 
$\A_{{\rm spin}(\aleph_0)}$ is established in \cite{arveson--price2003}, while the uniqueness of $\A_{{\rm spin}(\aleph_0)}$ up to a
C$^*$-isomorphism that sends universal spin unitaries to universal spin unitaries is a consequence of the universal property of universal C$^*$-algebras.
Likewise, each $\A_{{\rm spin}(m)}$ can be realised as a unital C$^*$-algebra of $\A_{{\rm spin}(\aleph_0)}$ be selecting any $m$ of the universal spin unitaries that
generate $\A_{{\rm spin}(\aleph_0)}$.

Therefore, without loss of generality, we may assume $\{\mathfrak u_k\}_{k\in\mathbb N}$ is a prescribed set of universal spin unitaries and that
$\A_{{\rm spin}(m)}$ is the unital C$^*$-subalgebra of $\A_{{\rm spin}(\aleph_0)}$ generated by $\mathfrak u_1,\dots,\mathfrak u_m$.
Furthermore, if $\mathcal O_{{\rm spin}(\aleph_0)}$ is the operator subsystem of $\A_{{\rm spin}(\aleph_0)}$ generated by (i.e., spanned by)
$\{\mathfrak u_k\}_{k\in\mathbb N}$,
then we may view $\mathcal O_{{\rm spin}(m)}$ as the operator subsystem of $\mathcal O_{{\rm spin}(\aleph_0)}$ generated by 
$\mathfrak u_1,\dots,\mathfrak u_m$. Because $\mathcal O_{{\rm spin}(\aleph_0)}$ is not norm closed, it is also of interest to consider the operator system
$\overline{\mathcal O_{{\rm spin}(\aleph_0)}}$, the norm-closure of $\mathcal O_{{\rm spin}(\aleph_0)}$ in $\A_{{\rm spin}(\aleph_0)}$.

Our goals are to prove the following two theorems.

\begin{theorem}\label{inf1} If $\nu=(v_1,v_2,\dots)$ is a sequence of selfadjoint anticommuting unitary operators acting on a Hilbert space $\H$, if 
$\mathcal O_\nu=\mbox{\rm Span}\{v_k\,|\,k\in\mathbb N\}$ and if $\overline{\mathcal O_\nu}$ is the norm-closure of $\mathcal O_\nu$ in $\B(\H)$, then
the linear map $\phi:\mathcal O_{{\rm spin}(\aleph_0)}\rightarrow \mathcal O_\nu$ in which $\phi(\mathfrak u_k)=v_k$, for every $k\in\mathbb N$, 
is a unital complete order isomorphism. Moreover, there exists a unital complete order isomorphism
$\Phi:\overline{\mathcal O_{{\rm spin}(\aleph_0)}}\rightarrow \overline{\mathcal O_\nu}$ such that $\phi=\Phi_{\vert \mathcal O_{{\rm spin}(\aleph_0)}}$.
\end{theorem}

\begin{theorem}\label{inf2} $\cstare(\mathcal O_{{\rm spin}(\aleph_0)})\cong \cstare(\overline{\mathcal O_{{\rm spin}(\aleph_0)}})\cong\displaystyle\bigotimes_1^\infty \M_2(\mathbb C)$.
\end{theorem}

We now move to the proof of each of these two results.

\begin{proof} (of Theorem \ref{inf1}). We first note that the countable set $\{\mathfrak u_k\}_{k\in\mathbb N}$ is a linear basis for $\mathcal O_{{\rm spin}(\aleph_0)}$.
As it is clearly a spanning set,
we need only note that its elements are linearly independent. To this end, select a finite subset $\mathcal F$ of $\{\mathfrak u_k\}_{k\in\mathbb N}$, and let $m$ be the maximum of the k$\in\mathbb N$ 
for which $v_k\in\mathcal F$. Thus, the elements of $\mathcal F$ belong to the set of generators of the universal C$^*$-algebra $\A_{{\rm spin}(m)}$. Therefore, the elements of $\mathcal F$
may be realised as spin unitary matrices, and we showed in Proposition \ref{prop:elementary} that such unitary matrices must be linearly independent. 

Likewise, any finite subset $\mathcal G$ of $\{v_k\,|\,k\in\mathbb N\}$ determines a finite-dimensional 
spin operator system (say of dimension $\ell+1$) which must, by Theorem \ref{main result}, be completely order isomorphic to 
$\A_{{\rm spin}(\ell)}$ via an isomorphism that takes spin unitaries to spin unitaries. Hence, the elements of $\mathcal G$ are linearly independent, proving that 
$\{v_k\}_{k\in\mathbb N}$ is a linear basis for $\mathcal O_\nu$.

The  linear map $\phi$ is obtained by restricting the unital
C$^*$-algebra homomorphism $\pi:\A_{{\rm spin}(\aleph_0)}\rightarrow\B(\H)$, in which $\pi(\mathfrak u_k)=v_k$ for every $k\in\mathbb N$, to the operator system
$\mathcal O_{{\rm spin}(\aleph_0)}$. (The existence of $\pi$
is a consequence of the universal property of $\A_{{\rm spin}(\aleph_0)}$.) Hence, $\phi$ is unital and completely positive; as it takes a linear basis of
$\mathcal O_{{\rm spin}(\aleph_0)}$ to a linear basis of $\mathcal O_\nu$, the map $\phi$ is also a bijection. All that remains is to show that the linear inverse $\phi^{-1}$
is completely positive. 
To this end, let $Y\in\M_n(\mathcal O_\nu)_+$; thus, there exists a unique $X\in \M_n(\mathcal O_{{\rm spin}(\aleph_0)})$ such that
$Y=\phi^{(n)}(X)$. We aim to show that
$X=(\phi^{-1})^{(n)}(Y)$ is a positive element of $\M_n(\mathcal O_{{\rm spin}(\aleph_0)})$.
To do so, note that each entry $x_{ij}$ of $X$ is a linear combination of a finite set $\mathcal F_{ij}$ of unitaries in $\{\mathfrak u_k\,|\,k\in\mathbb N\}$. Let $m$ be the
maximum of all $k\in\mathbb N$ for which $v_k\in\mathcal F_{ij}$ for some $i$ and $j$. Thus, we may view $X$ as being an element of $\M_n(\mathcal O_{{\rm spin}(m)})$.
The restriction $\phi_m$ of $\phi$ to $\mathcal O_{{\rm spin}(m)}$ is a unital complete order embedding, by Theorem \ref{main result}, which implies that 
$Y=\phi_m^{(n)}(X)$ is positive in $\M_n(\mathcal O_\nu)$ only if $X$ is positive in $\M_n(\mathcal O_{{\rm spin}(m)})$. However, if $X$ is positive in $\M_n(\mathcal O_{{\rm spin}(m)})$,
then it is also positive in $\M_n(\mathcal O_{{\rm spin}(\aleph_0)})$, which completes the proof that $\phi$ is a unital complete order isomorphism.

The norm of an element $x$ in an operator system $\osr$ is determined by
\[
\|x\|=\inf\left\{t>0\,|\,\left[\begin{array}{cc} te_\osr & x \\ x^* & te_\osr\end{array}\right] \in \M_2(\osr)_+\right\},
\]
and, therefore, a unital completely positive map $\psi:\osr\rightarrow\ost$ is a complete order embedding if and only if $\psi$ is a complete isometry \cite[\S4]{choi--effros1977}.
Thus, the unital complete order isomorphism $\phi:\mathcal O_{{\rm spin}(\aleph_0)}\rightarrow \mathcal O_\nu$ is a complete isometry; hence, $\phi$ extends to a 
completely isometric unital bijection $\Phi:\overline{\mathcal O_{{\rm spin}(\aleph_0)}}\rightarrow \overline{\mathcal O_\nu}$, implying that $\Phi$ is a unital complete order isomorphism, by
\cite[\S4]{choi--effros1977}. This completes the proof of Theorem \ref{inf1}.
\end{proof}

We note that the proof of Theorem \ref{inf1} may be adapted so that it applies, as well, to uncountable sets of spin unitaries acting on infinite-dimensional Hilbert spaces.

\begin{proof} (of Theorem \ref{inf2}). There exists a sequence $\nu=(v_1,v_2,\dots)$ of spin unitaries acting on a separable Hilbert space $\H$ such that the C$^*$-algebra
generated by the operator system $\mathcal O_v$ is (isomorphic to) the CAR algebra
$\displaystyle\bigotimes_1^\infty \M_2(\mathbb C)$ \cite{arveson--price2003}. Thus, the C$^*$-envelope of
$\mathcal O_\nu$ is a quotient of the CAR algebra. However, because the CAR algebra is simple, there are no nontrivial quotients of it; hence,
$\cstare(\mathcal O_\nu)\cong \displaystyle\bigotimes_1^\infty \M_2(\mathbb C)$.
Now because Theorem \ref{inf1} asserts that $\mathcal O_\nu$ and $\mathcal O_{{\rm spin}(\aleph_0)}$ are unitally complete order isomorphic, they necessarily have the same 
C$^*$-envelopes, which yields $\cstare(\mathcal O_{{\rm spin}(\aleph_0)})\cong \displaystyle\bigotimes_1^\infty \M_2(\mathbb C)$.
As the same argument applies to the operator system $\overline{\mathcal O_{{\rm spin}(\aleph_0)}}$, the proof of Theorem \ref{inf2} is complete.
\end{proof}

\section*{Acknowledgements}
This work was supported, in part, by the NSERC Undergraduate Student Research Award and NSERC Discovery Grant programs, the 
Canada Foundation for Innovation, and the Canada Research Chairs program.


\bibliographystyle{amsplain}
\bibliography{doug-refs}

\end{document}